\documentclass[12pt]{article}
\usepackage{e-jc}

\usepackage{amsmath}
\usepackage{amssymb}
\usepackage{exscale}
\usepackage{amsthm}
\usepackage{epsfig}
\usepackage{verbatim}
\usepackage{setspace}

\DeclareMathOperator{\dis}{dis}

\DeclareMathOperator{\vol}{\lambda^d}

\DeclareMathOperator{\fs}{ref}

\theoremstyle{plain}
\newtheorem{theorem}{Theorem}[section]
\newtheorem{proposition}[theorem]{Proposition}
\newtheorem{lemma}[theorem]{Lemma}

\theoremstyle{definition}

\newcommand{\setmid}{\,|\,}
\newcommand{\N}{{\mathbb{N}}}

\newcommand{\R}{{\mathbb{R}}}

\newcommand{\dcubes}{{\mathcal{C}_d}}

\newcommand{\dboxes}{{\mathcal{R}_d}}
\newcommand{\Nb}{N_{[\,\,]}}

\title{Construction of Minimal Bracketing Covers for Rectangles}

\author{$\quad$Michael Gnewuch 
\\$\quad$\\ 
{\small Department of Computer Science,
  Kiel University}\\
{\small Christian-Albrechts-Platz 4, 24098 Kiel, Germany}\\
{\small email: mig@informatik.uni-kiel.de}}

\date{\dateline{Sept 5, 2007}{Jul 16, 2008}{Jul 21, 2008}}

\begin{document}

\maketitle

\begin{abstract}
We construct explicit 
$\delta$-bracketing covers with minimal cardinality for the set 
system of (anchored) rectangles in the two dimensional unit cube. 
More precisely, the cardinality of these 
$\delta$-bracketing covers are bounded from
above by $\delta^{-2} + o(\delta^{-2})$. A lower bound for the 
cardinality of arbitrary $\delta$-bracketing covers for $d$-dimensional
anchored boxes from [M. Gnewuch, Bracketing numbers for axis-parallel boxes and applications to
geometric discrepancy, J.~Complexity 24 (2008) 154-172] implies the lower bound
$\delta^{-2}+O(\delta^{-1})$ in dimension $d=2$, showing that
our constructed covers are (essentially) optimal.

We study also other $\delta$-bracketing covers for the set system of
rectangles, deduce the coefficient of the most significant term 
$\delta^{-2}$ in the asymptotic expansion of their cardinality, 
and compute their cardinality for explicit 
values of $\delta$.
\end{abstract}

%-------------------------------------------------------------------
% INTRODUCTION
%-------------------------------------------------------------------

\section{Introduction}
\label{INT}

Entropy numbers are measures of the size of a given class $\mathcal{F}$
of functions or sets  and they are frequently used in
fields like density estimation, empirical processes 
or machine learning.
Good bounds for these entropy numbers, in particular the covering
or the bracketing numbers, can, e.g., be used to prove bounds
on the expectations of suprema of empirical processes (as, e.g.,
Dudley's metric entropy bound), concentration of measure
results for these suprema, or to verify that a class $\mathcal{F}$
of functions
or sets is a Glivenko-Cantelli or Donsker Class, i.e., that the 
corresponding $\mathcal{F}$-indexed empirical process $\mathbb{G}_n$
exhibits a certain convergence behavior as $n$ tends to infinity
(cf. \cite{DeL, Tal, vdVW}).  

They are also useful in geometric discrepancy theory, i.e.,
in the theory of uniform distribution. (Different facets of this 
theory are nicely described in the monographs 
\cite{BeC, Cha, DrT, Mat, Nie}.)
In geometric discrepancy theory
one tries to distribute $n$ points in a way to minimize the 
``discrepancy'' between a given (probability) measure and the
measure induced by the points (each point has mass $1/n$)
with respect to some class of measurable sets $\mathcal{C}$.
If one takes, e.g., the class 
$\mathcal{C}_d := \{ \prod^d_{i=1} [0,x_i) \,|\, 
x_1,\ldots,x_d \in [0,1]\}$ of anchored $d$-dimensional axis-parallel
boxes, the Lebesgue measure $\lambda^d$ on
$[0,1]^d$, and an $n$-point set $P\subset [0,1]^d$, then the
so-called star discrepancy of $P$
\begin{equation*}
d^*_{\infty}(P) := \sup_{C\in\mathcal{C}_d}
\left| \lambda^d(C) - \frac{1}{n} |P\cap C| \right|
\end{equation*}
is a measure of how uniform the points of $P$ are distributed
in $[0,1]^d$; here $|P\cap C|$ denotes the cardinality of the set $P\cap S$. 
If one substitutes the set system $\mathcal{C}_d$ 
by, e.g., the system of all $d$-dimensional axis-parallel boxes
$\mathcal{R}_d := \{ \prod^d_{i=1} [x_i, y_i) \,|\, 
x_1,y_1,\ldots,x_d, y_d \in [0,1]\}$, one gets another measure 
of uniformity, the so-called extreme discrepancy
\begin{equation*}
d^{\rm e}_{\infty}(P) := \sup_{C\in\mathcal{R}_d}
\left| \lambda^d(C) - \frac{1}{n} |P\cap C| \right|.
\end{equation*} 
Certain types of discrepancy are intimately related to 
multivariate numerical
integration of certain function classes (see, e.g., \cite{ 
Cha, DrT, HSW, Mat, Nie, NoW}); a well-known result in this direction
is the Koksma-Hlawka inequality which, written as an equality, reads
\begin{equation*}
\sup_{f\in B} \left| \int_{[0,1]^d} f(x)\,{\rm d}x
- \frac{1}{n} \sum^n_{i=1} f(t_i) \right|
= d^*_{\infty}(t_1,\ldots,t_n),
\end{equation*}
where $B$ is the unit ball in some particular Sobolev space
of functions (see, e.g., \cite{HSW}).

Thus for multivariate numerical integration it is desirable to be 
able to calculate the
star discrepancy of a given point configuration $\{t_1,\ldots,t_n\}$,
to have (useful) bounds on the smallest possible discrepancy of 
any $n$-point set, and to be able to construct 
sets satisfying such bounds.

Algorithms approximating the star discrepancy of a given $n$-point
set up to some admissible error $\delta$ with the help of bracketing covers have been 
provided in \cite{Thi0,Thi} (see also the discussion in \cite{Bra}).
The more efficient algorithm from \cite{Thi} generates $\delta$-bracketing covers 
of $\mathcal{C}_d$ (for a rigorous definition
see Sect.~\ref{PRE}) and uses those to test the 
discrepancy of a given point set.
The last step raises the task of orthogonal range counting. Depending whether
the orthogonal range counting is done in a naive way or (in small dimensions)
by employing data structures based on range trees, the running time of the 
algorithm is of order
\begin{equation*}
O(dn|\mathcal{B}_\delta|) \hspace{3ex}\text{or}\hspace{3ex}
O\left( (d+(\log n)^d) |\mathcal{B}_\delta| + C^d n (\log n)^d \right),
\end{equation*}
where $\mathcal{B}_\delta$ is the generated $\delta$-bracketing cover and 
$C>1$ some constant. The cost of generating the $\delta$-bracketing cover
$\mathcal{B}_\delta$ is obviously a lower bound for the running time of the 
algorithm and is of order $\Omega(d|\mathcal{B}_\delta|)$.
Thus the running time of the algorithm from \cite{Thi} depends linear on the size of the
generated bracketing covers.

Bounds on the smallest possible star discrepancy with essentially
optimal asymptotic behavior for fixed dimension $d$ have been
known for a long time (see, e.g, \cite{BeC, Cha, DrT, Mat, Nie}). 
Nevertheless, they are nearly useless
for high-dimensional numerical integration, because one needs 
exponentially many sample points in $d$ to reach the asymptotic 
range. Starting with the paper \cite{HNWW} probabilistic approaches
have been used to prove bounds for the star, the extreme, and other types
of discrepancy that are useful for samples of 
moderate size \cite{Dic, DoG, DGKP,DGS, AvLpEx, Bra, HSW, Mha}.
In particular, these investigations focused on the explicit dependence on the
number of points $n$ \emph{and} on the dimension $d$.
(Of course, probabilistic approaches had been used in discrepancy
theory before \cite{HNWW}, see, e.g., \cite{Bec, BeC}.
But these studies had not explored the explicit dependence
on the dimension $d$.)

Let us describe these results in more detail:
We denote the smallest possible star discrepancy of any $n$-point
configuration in $[0,1]^d$ by 
\begin{equation*}
d^*_\infty(n,d) = 
\inf_{P\subset [0,1]^d; |P|=n} d^*_\infty(P)
\end{equation*}
and the so-called inverse of the star discrepancy by 
\begin{equation*}
n^*_\infty(\varepsilon, d) =
\min\{n\in\N \setmid d^*_\infty(n,d) \leq \varepsilon\}\,.
\end{equation*}

In \cite{HNWW} Heinrich, Novak, Wasilkowski, and Wo\'zniakowski
proved the bounds
\begin{equation}
\label{hnww}
d^*_\infty(n,d) \leq C\sqrt{\frac{d}{n}}
\hspace{2ex}\text{and}\hspace{2ex}
n^*_\infty(\varepsilon, d) \leq \lceil C^2 d \varepsilon^{-2}\rceil\,,
\end{equation}
where $C$ is a universal constant. 
The proof uses a theorem
of Talagrand on empirical processes \cite[Thm. 6.6]{Tal}
combined with an upper bound of Haussler on 
%$L^1$-packing numbers 
so-called covering numbers of Vapnik-\v{C}ervonenkis classes
\cite{Hau}. (Since the theorem of Talagrand holds not only
under a condition on the covering number of the set system
$\mathcal{S}$ under consideration, but also under
the alternative condition that the $\delta$-bracketing
number of $\mathcal{S}$ is bounded 
from above by $(C \delta^{-1})^d$, $C$ some 
constant \cite[Thm. 1.1]{Tal}, one can reprove (\ref{hnww}) by 
using the bracketing result \cite[Thm.~1.15]{Bra}
instead of the result of Haussler.)

An advantage of (\ref{hnww}) is that the dependence of the inverse
of the discrepancy on $d$ is optimal. This was verified in \cite{HNWW}
by a lower bound for the inverse, which was improved by Hinrichs 
\cite{Hin} to $n^*_\infty(d,\varepsilon) \geq c_0 d \varepsilon^{-1}$.
A disadvantage of (\ref{hnww}) is that so far no good 
estimate for the constant $C$ has been published.

An alternative approach via using bracketing covers and large deviation
inequalities of Chernov-Hoeffding type leads to slightly worse bounds
with explicitly given small constants \cite{Dic, DoG, DGKP, DGS, Bra, HNWW}. 
Let $N_{[\,]}(\mathcal{C}_d, \delta)$
denote the bracketing number, i.e., the cardinality of a minimal 
$\delta$-bracketing cover of $\mathcal{C}_d$. Then
\begin{equation}
\label{nstar}
n^*_\infty(\varepsilon, d) \le
\left\lceil \frac{2}{\varepsilon^2} \left(\ln N_{[\,]}(\mathcal{C}_d, \varepsilon/2)
+ \ln 2 \right) \right\rceil,
\end{equation}
see \cite[Proof of Thm.~3.2]{DGS}. Thus improved bounds of the bracketing
entropy  $\ln N_{[\,]}(\mathcal{C}_d, \delta)$ would lead directly to improved bounds on
the inverse of the star discrepancy and of the star disprepancy as well (although its dependence
on the entropy cannot be expressed by an explicit formula like (\ref{nstar}), since the 
corresponding parameter $\delta$ should be chosen to be of the order of the star 
discrepancy; see again \cite[Proof of Thm.~3.2]{DGS}). 
 
Attempts have been made to provide deterministic 
algorithms constructing point sets whose star discrepancy
satisfies the probabilistic bounds resulting from this alternative approach
\cite{DoG, DGKP, DGS}. The running times of the algorithms depend
on the cardinality of suitable $\delta$-bracketing covers;
smaller covers would reduce the running times.

These examples show that for discrepancy theory and its application
to multivariate numerical integration it is of interest
to be able to construct minimal bracketing covers.

In \cite[Thm. 2.7]{DGS} we derived for fixed dimension $d$ 
the upper bound
\begin{equation}
N_{[\,]}(\mathcal{C}_d, \delta) \le \frac{d^d}{d!} \delta^{-d} 
+ O(\delta^{-d+1})
\end{equation}
for the bracketing number of the set system $\mathcal{C}_d$. 
In \cite{Bra} the bounds
\begin{equation}
\label{lobo}
\delta^{-d}(1-c_d \delta) \le N_{[\,]}(\mathcal{C}_d, \delta) 
\le 2^{d-1}\frac{d^d}{d!} (\delta+1)^{-d}, 
\end{equation}
where $c_d$ depends only on the dimension $d$, where proved.
Obviously there is a gap between the upper bounds and the lower
bound. In this paper we prove that in dimension $d=2$ the lower
bound is sharp. More precisely, we construct explicit 
$\delta$-bracketing covers $\mathcal{R}_\delta$ 
whose cardinality is bounded from
above by $\delta^{-2} + o(\delta^{-2})$; thus $1$ is the correct
coefficient in front of the most significant term in the 
expansion of the bracketing number 
$N_{[\,]}(\mathcal{C}_d, \delta)$ with respect to $\delta^{-1}$.
Furthermore, we discuss other constructions in dimension $d=2$
(e.g., the cover from \cite{Thi}) 
and compare them.
We conjecture that the lower bound in (\ref{lobo}) is sharp in 
the sense that $N_{[\,]}(\mathcal{C}_d, \delta) = \delta^{-d} 
+ o_d(\delta^{-d})$ holds for all $d$; here $o_d$ should emphasize 
that the implicit constants in the $o$-notation may depend on $d$.
We are convinced that this upper bound can be proved constructively 
by extending the ideas we used to generate $\mathcal{R}_\delta$
to higher dimensions.

%-------------------------------------------------------------------
% DEFINITION `BRACKETING' PLUS SOME MORE ELEMENTARY RESULTS
%-------------------------------------------------------------------

\section{Preliminaries}
\label{PRE}

Let $d\in \N$ and put $[d]:=\{1,\ldots,d\}$.
For $x, y \in [0,1]^d$
we write $x \le y$ if $x_i \le y_i$ holds for all $i \in [d]$. We write $[x,y] :=
\prod_{i \in [d]} [x_i, y_i]$ and use corresponding notation for open and
half-open intervals. We put $V_x :=
\vol([0,x])$ and $V_{x,y} := \vol([x,y])$, where $\vol$ is the 
$d$-dimensional Lebesgue measure. Similarly, we put $V_A := 
\vol(A)$ for any measurable subsets $A$ of $[0,1]^d$.
In this paper we consider the classes 
$$\dcubes = \{[0,x) \setmid x\in [0,1]^d\}\hspace{2ex}\text{and} 
\hspace{2ex}\dboxes = \{[x,y) \setmid x,y\in [0,1]^d\}$$ 
of subsets of $[0,1]^d$. The elements of $\dcubes$ are called
\emph{anchored
(axis-parallel) boxes} or simply \emph{corners}. 
The elements of $\dboxes$ are called \emph{unanchored
(axis-parallel) boxes}. (Here the word ``unanchored''
is of course meant in the sense of ``not necessarily anchored''.)

Let $\mathcal{F}\in\{\dcubes, \dboxes\}$. For a given 
$\delta \in (0,1]$ and $A,B \in \mathcal{F}$ with $A\subseteq B$ 
we call the set
\begin{equation*}
[A,B]_{\mathcal{F}} :=
\{ C\in\mathcal{F} \setmid A\subseteq C \subseteq B\}
\end{equation*}
a \emph{$\delta$-bracket} of $\mathcal{F}$ if its 
\emph{weight} $W([A,B])$ defined by
\begin{equation*}
W([A,B]) := V_B - V_A 
\end{equation*}
does not exceed $\delta$.
A \emph{$\delta$-bracketing cover} of $\mathcal{F}$
is a set of $\delta$-brackets whose union is $\mathcal{F}$. 
By $\Nb(\mathcal{F},\delta)$ we denote the \emph{bracketing
  number} of $\mathcal{F}$, i.e., the smallest number of  
$\delta$-brackets whose
union is $\mathcal{F}$. The quantity $\ln \Nb(\mathcal{F},\delta)$
is called the \emph{bracketing entropy} of $\mathcal{F}$.
In \cite{Bra} we showed in particular that
\begin{equation}
\label{triest}
\Nb(\dcubes,\delta) \leq \Nb(\dboxes, \delta)
\le (N(\dcubes,\delta/2))^2. 
\end{equation}
The second inequality was verified by using arbitrary 
$\delta/2$-bracketing covers of $\dcubes$ of cardinality $\Lambda$
to construct 
$\delta$-bracketing covers of $\dboxes$ of cardinality at most
$\Lambda^2$  
(cf. \cite[Lemma 1.18]{Bra}); that is why we can restrict ourselves
to the construction of bracketing covers of $\dcubes$.

Let us identify the boxes $[0,x)$ in $\dcubes$ with their
right upper corners $x\in [0,1]^d$. 
According to this convention, we identify the bracket 
$[[0,x), [0,y)]_{\dcubes}$ with the $d$-dimensional box
$[x,y]$. 

If we are interested in $\delta$-bracketing covers of $\dcubes$
with small 
cardinality it is clear that we should try to maximize
the volume of the $\delta$-brackets used. 
The following lemma states how $\delta$-brackets of $\dcubes$
with maximum volume look like.

\begin{lemma}
\label{OptBox}
Let $d\ge 2$, $\delta\in(0,1)$, and let $z\in [0,1]^d$ with
$V_z > \delta$. Put 
\begin{equation*}
%\label{xzdelta}
x=x(z,\delta) := \left( 1 - \frac{\delta}{V_z}
\right)^{1/d}z.
\end{equation*}
Then $[x,z]$ is the uniquely determined $\delta$-bracket  having
maximum volume of all $\delta$-brackets of 
$\dcubes$ that contain $z$.
Its volume is 
\begin{equation*}
V_{x,z} = \left( 1 - \left( 1-\frac{\delta}{V_z}
\right)^{1/d} \right)^d V_z.
\end{equation*}
\end{lemma} 

(In the case where $V_z \le \delta$ it is easy to see that $z$ is 
always contained in some $\delta$-bracket $[0,\zeta)$ 
with maximum volume $V_{\zeta} = \delta$.) 
For a proof of the lemma see \cite[Lemma 1.1]{Bra}.

Now we state a ``scaling lemma'' which we shall use frequently
throughout the paper.

\begin{lemma}
\label{scaling}
Let $\delta \in (0,1)$ and $\lambda = (\lambda_1,\ldots,\lambda_d) 
\in (0,\infty)^d$. 
Let 
\begin{equation*}
\Phi(\lambda): \R^d \to \R^d\,,\, (x_1,\ldots,x_d) \mapsto
(\lambda_1x_1,\ldots,\lambda_dx_d).
\end{equation*} 
Furthermore, let $S \subseteq [0,1]^d$ such that 
$\Phi(\lambda)S\subseteq [0,1]^d$. Then the smallest number of 
$\delta$-brackets whose union covers $S$ is the smallest number of
$((\prod^d_{i=1}\lambda_i)\delta)$-brackets whose union covers 
$\Phi(\lambda)S$.
\end{lemma}

The proof is obvious since scaling a bracket by 
applying $\Phi(\lambda)$ implies that its weight is scaled by 
the multiplicative
factor $\prod^d_{i=1}\lambda_i$.

Let us briefly recapitulate the construction of a 
$\delta$-bracketing cover $\mathcal{G}_\delta$ from 
\cite{DGS} in which the $\delta$-brackets are the cells
in a non-equidistant grid. We do so for two reasons: 
We want to compare the cardinality of $\mathcal{G}_\delta$
with the (more sophisticated) bracketing covers we present later,
and, what is more important, the construction of 
$\mathcal{G}_\delta$ can be viewed as a ``building block'' of all
these bracketing covers.

We construct the non-equidistant grid
\begin{equation}
\label{non_eq_grid}
\Gamma_\delta = 
(\{x_0, x_1, ...,x_{\kappa(\delta,d)}\}\cup \{0\})^d\,,
\end{equation} 
where $x_0, x_1,..., x_{\kappa(\delta,d)}$ is a decreasing sequence in
$(0,1]$. We calculate this sequence recursively in the following way:
Put $x_0 := 1$ and $x_1 := (1-\delta)^{1/d}$. If $x_i > \delta$, 
then define
$x_{i+1} := (x_i - \delta)x^{1-d}_1$.
If $x_{i+1} \leq \delta$, then
put $\kappa(\delta,d) := i+1$, 
otherwise proceed by calculating $x_{i+2}$.

Since $\mathcal{G}_\delta$ consists of the cells of $\Gamma_{\delta}$,
i.e., of all closed $d$-dimensional boxes $B$ whose intersection with 
$\Gamma_{\delta}$ consists exactly of the $2^d$ corners of $B$, we have 
\begin{equation}
\label{cardG}
|\mathcal{G}_\delta| = (\kappa(\delta,d) + 1)^d.
\end{equation}
It was shown in \cite{DGS}, that $\mathcal{G}_\delta$ is a bracketing
cover (without explicitly using this notion) and that 
\begin{equation}
\label{kappa}
\kappa(\delta, d) = \left\lceil
\frac{d}{d-1}\frac{\ln(1-(1-\delta)^{1/d}) -
  \ln(\delta)}{\ln(1-\delta)} \right\rceil\,.
\end{equation}
Furthermore, it was shown that the inequality 
$\kappa(\delta, d) \leq \big\lceil \frac{d}{d-1} \frac{\ln(d)}{\delta} 
\big\rceil$
holds, and that the quotient of the left and the right hand side of
this inequality converges to $1$ as $\delta$ approaches $0$.
But to make proofs shorter in what follows, it is better to use the
more precise estimate
\begin{equation}
\label{kappaO1}
\kappa(\delta,d) = \frac{d}{d-1}\ln(d)\delta^{-1} + O(1)
\hspace{2ex}\text{as $\delta$ approaches zero.}
\end{equation}  
It follows directly from the following identities which are easy to 
check:
\begin{equation}
\label{ln1}
(\ln(1-\delta))^{-1} = -\delta^{-1} + O(1)
\end{equation}
and 
\begin{equation}
\label{ln2}
\ln(1-(1-\delta)^{1/d}) = \ln(\delta)-\ln(d) + O(\delta)
\end{equation}
as $\delta$ tends to zero.

Let us now confine ourselves to dimension $d=2$ and use the 
shorthand $\kappa(\delta)$ for $\kappa(\delta,2)$.
Put $a_i(\delta) := (1-i\delta)^{1/2}$ for 
$i=0,1,...,\lceil \delta^{-1} \rceil -1$.
Then in fact, $\kappa(\delta)+1$ is the minimal number of 
$\delta$-brackets of heights $1-a_1(\delta)$ whose union covers
the stripe $[(0, a_1(\delta)), (1,1)]$; the $\delta$-brackets
covering the stripe are the rectangles $[(x_1, a_1(\delta)), (x_0,1)]$,
$[(x_2, a_1(\delta)), (x_1,1)]$,...,
$[(0, a_1(\delta)), (x_{\kappa(\delta)},1)]$.

Let us more generally define $\omega(\delta,t)$ to be the minimal
number of $\delta$-brackets of heights $1-a_1(\delta)$ whose 
union covers the stripe $[(t, a_1(\delta)), (1,1)]$ for some
$t\in [0,1]$. We calculate again $x_0, x_1,\ldots$
as above and determine $\omega(\delta,t)$ such that
$x_{\omega(\delta,t)-1} > t$ and 
$[(x_1, a_1(\delta)), (x_0,1)]$,
$[(x_2, a_1(\delta)), (x_1,1)]$,...,
$[(t, a_1(\delta)), (x_{\omega(\delta,t)-1},1)]$
are $\delta$-brackets whose union covers the stripe 
$[(t, a_1(\delta)), (1,1)]$.
From the construction of the $x_i$ we see that
$$
x_i = (1-\delta)^{-i/2} - \delta(1-\delta)^{-1/2}\,
\frac{1-(1-\delta)^{-i/2}}{1-(1-\delta)^{-1/2}}
$$
and that $x_{i+1} \le t$ is satisfied if and only if
$$
i+1 \ge
2\,\frac{\ln \big( 1-(1-\delta)^{1/2} \big) - 
\ln \big( t(1-(1-\delta)^{-1/2}) + \delta(1-\delta)^{-1/2} \big)}
{\ln(1-\delta)}.
$$
Thus 
\begin{equation}
\label{omegadeltat}
\omega(\delta,t) = \left\lceil 2\;
\frac{\ln \big( 1-(1-\delta)^{1/2} \big) - 
\ln \big( t(1-(1-\delta)^{-1/2}) + \delta(1-\delta)^{-1/2} \big)}
{\ln(1-\delta)} \right\rceil.
\end{equation}
Observe that for $t=0$ we have indeed 
$\omega(\delta,0) = \kappa(\delta) + 1$. 
We shall use the numbers $\omega(\delta,t)$ for different $\delta$ 
and $t$ to
show that the last bracketing cover we present in this paper
exhibits the (asymptotically) optimal cardinality.

In the following three sections we present $\delta$-bracketing
covers with reasonably smaller cardinality than $\mathcal{G}_\delta$.

%-------------------------------------------------------------------
% CONSTRUCTION A (Thi\'emard)
%-------------------------------------------------------------------

\section{The Construction of Thi\'emard}

Before stating the algorithm of Thi\'emard to construct
a $\delta$-bracketing cover $\mathcal{T_\delta}$,
we want to explain its main idea in dimension $d=2$. (In 
\cite{Thi} the algorithm is discussed for arbitrary $d$.)

It covers $[0,1]^2$ successively with $\delta$-brackets
by decomposing all rectangles $P$ with weight $W(P) >\delta$
into smaller rectangles starting with the rectangle $[0,1]^2$.
More precisely, if $P$ is of the form $P=[\alpha,\beta]$ 
for some $\alpha=\alpha^P$, $\beta=\beta^P \in [0,1]^2$, 
then it calculates parameters $\gamma_1 = \gamma_1^P$,
$\gamma_2 = \gamma_2^P$ satisfying 
$\alpha_1 \le \gamma_1 \le \beta_1$ and 
$\alpha_2 \le \gamma_2 \le \beta_2$ and decomposes $P$ into
$$
Q^P_1 = [(\alpha_1, \alpha_2), (\gamma_1,\beta_2)]
\hspace{2ex}\text{and}\hspace{2ex}
P^P_1 = [(\gamma_1, \alpha_2), (\beta_1,\beta_2)].
$$
Afterwards it decomposes $P^P_1$ into 
$$
Q^P_2 = [(\gamma_1, \alpha_2), (\beta_1,\gamma_2)]
\hspace{2ex}\text{and}\hspace{2ex}
P^P_2 = [(\gamma_1, \gamma_2), (\beta_1,\beta_2)],
$$
resulting in the (almost disjoint) decomposition
$$
P = Q^P_1 \cup Q^P_2 \cup P^P_2.
$$
The right choice of $\gamma = (\gamma_1,\gamma_2)$ 
ensures $W(P^P_2) = \delta$ and $P^P_2$ is chosen to become an 
element of the final $\delta$-bracketing cover $\mathcal{T}_\delta$.

The rectangle $Q^P_1$ is of ``type 1'', the rectangle $Q^P_2$
of ``type 2'': if the algorithm decomposes them, then it chooses 
$\gamma_1^{Q^P_1} \in (\alpha_1^P, \gamma_1^P)$ and 
$\gamma_1^{Q^P_2} = \gamma_1^P$ implying that $Q^P_1$ will be 
decomposed into
three, but $Q^P_2$ only into two non-trivial rectangles.

That is why in the algorithm a rectangle $P$ is described by the
triple $(P, i, W(P))$, where $i\in\{1,2\}$ denotes the type of
the rectangle.

Denoted in pseudo-code, the algorithm looks as follows:

\vskip 0.2cm \noindent
{\bf Algorithm} {\sc THIEMARD}
\vskip 0.2cm \noindent
{\it Input:} $\delta \in (0,1)$.
\vskip 0.1cm \noindent
{\it Output:} A $\delta$-bracketing cover $\mathcal{T}_\delta$.
\vskip 0.2cm \noindent
{\it Main}
\vskip 0.2cm

$\mathcal{T}_\delta:=\emptyset$
\vskip 0.2cm

Decompose $([0,1]^2, 1, 1)$
\vskip 0.2cm \noindent
{\it Procedure {\sc decompose}} $(P,j,v)$
\vskip 0.2cm

Compute $\delta^P$ according to (\ref{deltap})
\vskip 0.2cm

Compute $\gamma^P$ according to (\ref{gammap})
\vskip 0.2cm

If $\delta^P v > \delta$
\vskip 0.2cm

\hspace{3ex}For $i$ from $j$ to $2$
\vskip 0.2cm

\hspace{6ex}Decompose $(Q^P_i, i, \delta^P v)$
\vskip 0.2cm

Else 
\vskip 0.2cm

\hspace{3ex}For $i$ from $j$ to $2$
\vskip 0.2cm

\hspace{6ex}$\mathcal{T}_\delta := \mathcal{T}_\delta \cup \{Q^P_i\}$
\vskip 0.2cm

$\mathcal{T}_\delta := 
\mathcal{T}_\delta \cup \{[\gamma^P, \beta^P]\}$ 
\vskip.2cm\noindent
For each triple $(P, j, v)$ we calculate $\delta^P \in (0,1)$ 
and $\gamma^P \in [0,1]^2$ as follows:
\begin{equation}
\label{deltap}
%\begin{split}
\delta^P = \left( \frac{\beta_1^P\beta_2^P - \delta}
{\beta_1^P\beta_2^P} \right)^{1/2}
\hspace{2ex}\text{if $j=1$,}\hspace{2ex}
\delta^P =  \frac{\beta_1^P\beta_2^P - \delta}
{\alpha_1^P\beta_2^P} 
\hspace{2ex}\text{if $j=2$,}
%\end{split}
\end{equation}
and
\begin{equation}
\label{gammap}
\gamma^P_i = 
\begin{cases}
\hspace{1ex}\alpha_i^P
\hspace{2ex}&\text{if $i<j$},\\
\hspace{0ex}\delta^P\beta_i^P
\hspace{2ex} &\text{if $i \ge j$.}\\
\end{cases}
\end{equation}

That the resulting set $\mathcal{T}_\delta$ is indeed a 
$\delta$-bracketing cover was proved in \cite{Thi}.
In Figure \ref{figt025} and \ref{figt005} one can see
the resulting cover $\mathcal{T}_\delta$ for $\delta = 0.25$
and $\delta = 0.05$.

\begin{figure}
\begin{center}
\hspace{1cm}{\epsfig{file=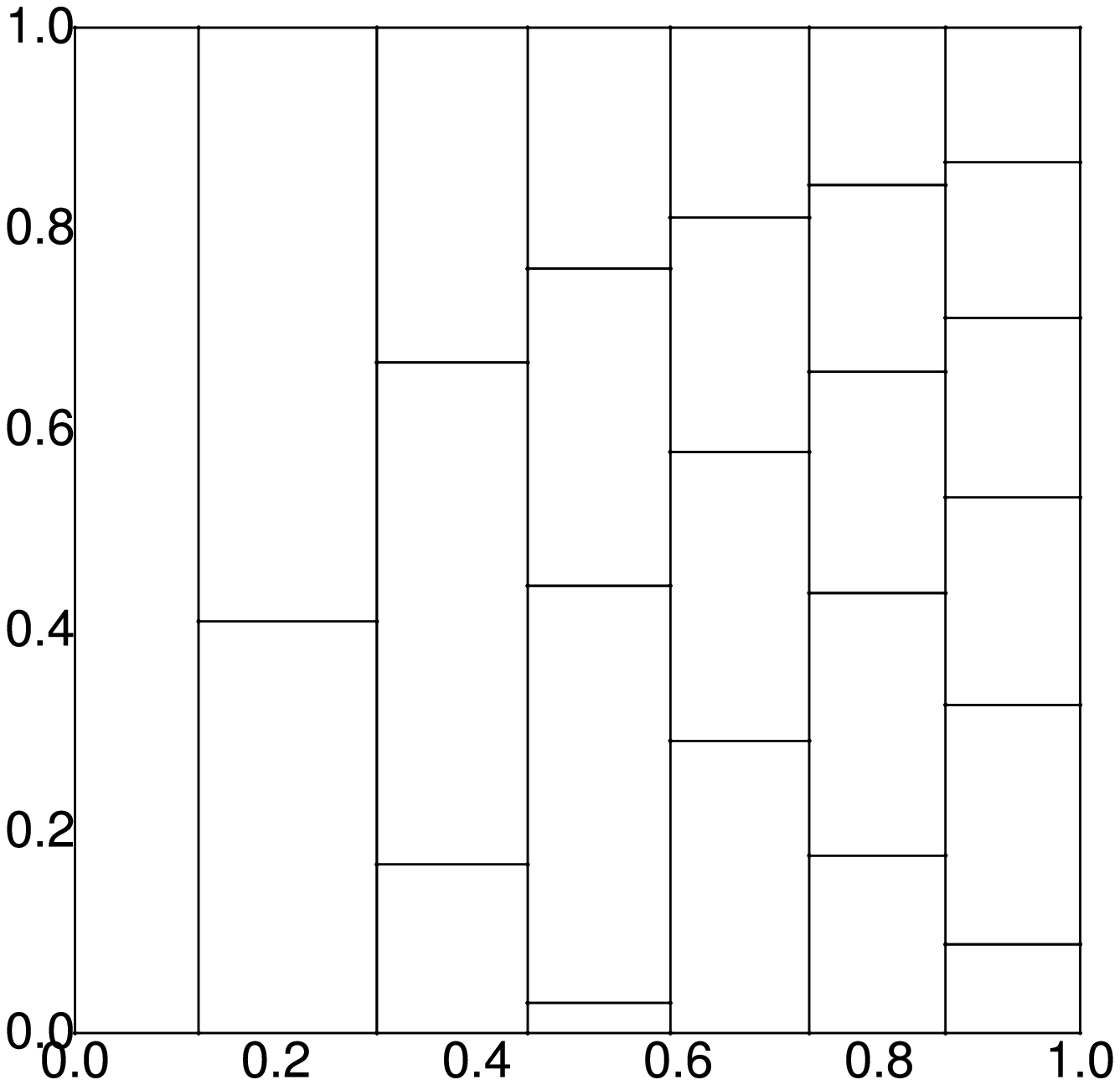,
height=.4\textheight
%, width=.45\textwidth
}}
\caption{\label{figt025}$\mathcal{T}_\delta$ for $\delta = 0.25$.}
\end{center}
\end{figure} 

\begin{figure}
\begin{center}
\hspace{1cm}{\epsfig{file=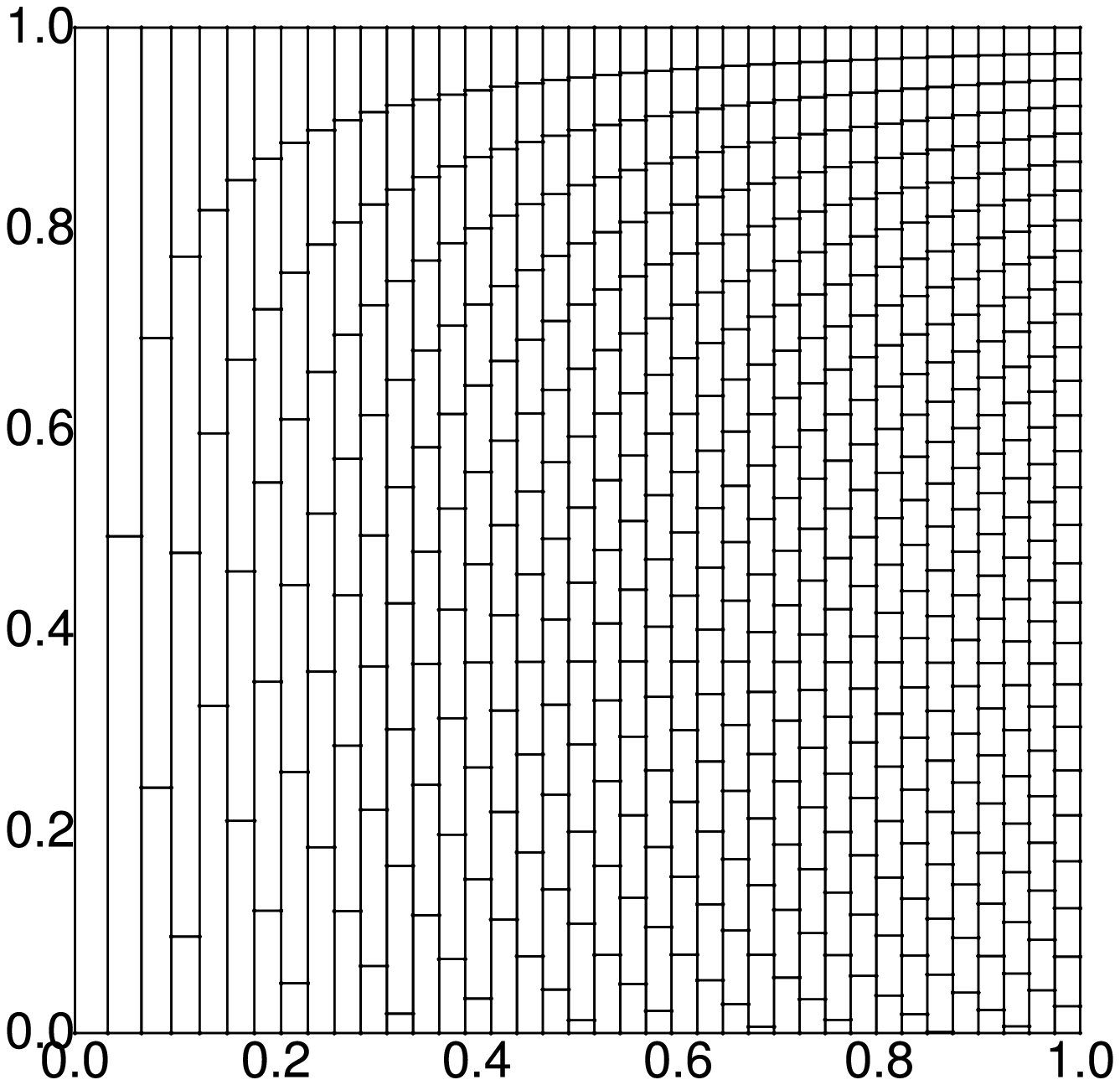,height
=.4\textheight
%,width=.45\textwidth
}}
\caption{\label{figt005}$\mathcal{T}_\delta$ for $\delta = 0.05$.}
\end{center}
\end{figure}  

Let us now determine the asymptotic behavior of 
$|\mathcal{T}_\delta|$ for $\delta$ tending to zero. 
%and calculate the coefficient in front of the dominating
%term $\delta^{-2}$.
In \cite[Theorem 3.4]{Thi} Thi\'emard proved the bound
\begin{equation*}
|\mathcal{T}_\delta| \le {2+h\choose 2}, \hspace{2ex}
\text{where} \hspace{2ex} 
h= \left\lceil \frac{2\ln(\delta)}{\ln(1-\delta)} \right\rceil.
\end{equation*}
This implies $|\mathcal{T}_\delta| \le 2(\ln(\delta^{-1}))^2\,
\delta^{-2} + o(\delta^{-2})$. We improve this estimate in the following
Proposition by deducing the correct asymptotic behavior in terms of
$\delta^{-1}$ and the
exact coefficient in front of the most significant term $\delta^{-2}$.

\begin{proposition}
For a given $\delta\in (0,1)$ we get
\begin{equation*}
|\mathcal{T}_\delta| = 2\ln(2)\delta^{-2}
+ O(\delta^{-1}).
\end{equation*}
\end{proposition}

\begin{proof}
From the discussion above (and also from Figure \ref{figt025} and 
\ref{figt005}) we see that Thi\'emard's algorithm decomposes
the unit rectangle $[0,1]^2$ into stripes 
$$
S^{(i)}_\delta := [(t_{i+1},0), (t_i,1)], \hspace{2ex}
i=0,\ldots, \tau(\delta), 
$$
and these stripes again into $\delta$-brackets; here the numbers
$t_i$ are the $x$-coordinates of the corners of all rectangles
of type $1$ that appear in the course of the algorithm.
More precisely, we have $t_0 = 1$, $t_{\tau(\delta)+1}=0$,
\begin{equation}
\label{tiplus1}
t_{i+1} = \left( 1 - \frac{\delta}{t_i} \right)^{1/2} t_i
= \left( \left( t_i - \frac{\delta}{2} \right)^2 - 
\frac{\delta^2}{4} \right)^{1/2},
\end{equation}
and $\tau(\delta)$ is uniquely determined by the relation
\begin{equation}
\label{unique}
0 < t_{\tau(\delta)} \le \delta.
\end{equation}
We have
\begin{equation}
\label{simest1}
t_i - \delta \le t_{i+1} < t_i - \frac{\delta}{2};
\end{equation}
both inequalities follow easily
from (\ref{tiplus1}).
From (\ref{unique}) and (\ref{simest1}) we get
\begin{equation}
\label{simest2}
\lceil \delta^{-1} \rceil -1 \le \tau(\delta) \le 
\lceil 2 \delta^{-1} \rceil - 1.
\end{equation}
Furthermore, we get by simple induction 
\begin{equation*}
t^2_{i+1} = 1 - \delta \sum^i_{k=0} t_k,
\end{equation*}
which, together with (\ref{unique}), results in
\begin{equation}
\label{keyineq}
\delta^{-1} - \delta \le \sum^{\tau(\delta)-1}_{k=0} t_k
< \delta^{-1}.
\end{equation}
Let us now calculate the number $s^{(i)}_\delta$ of 
$\delta$-brackets of widths $t_i-t_{i+1}$ that cover the stripe 
$S^{(i)}_\delta$.
Since the bracketing problem is symmetric in the $x$- and 
$y$-coordinate, we get from the discussion in the previous 
section
\begin{equation*}
s^{(0)}_\delta = \kappa(\delta) + 1,
\hspace{2ex}\text{where}\hspace{2ex}
\kappa(\delta) = \kappa(\delta,2)
\hspace{2ex}\text{as defined in (\ref{kappa}).}
\end{equation*}
(Note that $t_1 = (1-\delta)^{1/2}$.)
 From this we can derive $s^{(i)}_\delta$ for all 
$i \in \{0,\ldots, \tau(\delta)-1\}$ via ``scaling'':
Lemma \ref{scaling} gives us with the choice 
$\lambda = (t^{-1}_i,1)$
$$
s^{(i)}_\delta = \kappa(\delta/t_i) + 1
\hspace{2ex}\text{for all 
$i \in \{0,\ldots, \tau(\delta)-1\}$.}
$$
(Observe that $t_{i+1}/t_i = a_1(\delta/t_i)$.)
Furthermore, we have trivially $s^{(\tau(\delta))}_\delta = 1$.

%In \cite[Theorem 2.3]{DGS} we proved that 
%$\kappa_1(x) \le \lceil 2\ln(2)x^{-1} \rceil$ and that the quotient
%of the left and the right hand side of this inequality converges
%to $1$ as $\delta$ approaches $0$. Let us define
%$$
%R(x) = 2\,\left(
%\frac{\ln( 1- \sqrt{1-x} ) - \ln (x)}{\ln(1-x)} 
%- \frac{\ln(2)}{x} \right).
%$$
Then identity (\ref{kappaO1}) and the inequalities 
(\ref{simest2}) and (\ref{keyineq}) result in
\begin{equation*}
%\begin{split}
|\mathcal{T}_\delta| = \sum^{\tau(\delta)}_{i=0} s^{(i)}_\delta
= \sum^{\tau(\delta)-1}_{i=0} (\kappa(\delta/ t_i)+1) + 1
= 2\ln(2) \delta^{-1} \sum^{\tau(\delta)-1}_{i=0}  t_i + O(\delta^{-1})
= 2\ln(2)\delta^{-2} + O(\delta^{-1}).
%\end{split}
\end{equation*}

%An elementary analysis reveals that the function
%$F: ]0,1[\to\R$, $x \mapsto \delta R(x)$ is monotonic 
%decreasing with $\lim_{x\to 0} F(x) = 0$. Thus we have
%$R(x) = o(x^{-1})$.
%But from this, it seems, we cannot simply deduce
%$$
% \sum^{\tau(\delta)}_{i=0} |R(\delta/t_i)| 
%= o(\delta^{-1})  \sum^{\tau(\delta)}_{i=0} t_i,
%$$
%since the implicit constant in the relation
%$|t_i^{-1} R(\delta/ t_i)| = o(\delta^{-1})$ may depend on $i$.
%Thus we proceed in a more careful way.

%Put $\sigma_j = 1-j\delta/2$, $j= 0,\ldots, \lceil 2\delta^{-1} 
%\rceil - 1$.  
%According to (\ref{simest1}) we have $t_i \le \sigma_i$ for 
%$i=0, \ldots, \tau(\delta)$. Thus, with $N:=\lceil 2\delta^{-1} 
%\rceil - 1$,
%\begin{equation*}
%\delta \sum^{\tau(\delta)}_{i=0} \left|R\left(
%\frac{\delta}{t_i}\right)\right| 
%\le
%\sum^{\tau(\delta)}_{i=0} \left|\frac{\delta}{t_i}
%R\left(\frac{\delta}{t_i}\right)\right| 
%\le \sum^{N-2}_{i=0} \left|\frac{\delta}{\sigma_i}
%R\left(\frac{\delta}{\sigma_i}\right)\right|  
%\le \sum^{N}_{j=2} \left|\frac{2}{j}R\left(\frac2j
%\right)\right|.
%\end{equation*}
%If we put $C_j = |\frac2j R(\frac2j)|$, then $(C_j)$ is a null
%sequence and 
%$$
%\delta^2 \sum^{\tau(\delta)}_{i=0} \left| R \left(\frac{\delta}{t_i}
%\right) \right|
%\le \frac{2}{N} \sum^N_{j=2} C_j \to 0
%\hspace{2ex} \text{as $N\to\infty$,}
%$$
%due to the well-known fact that the Ces\`aro mean of an arbitrary 
%null sequence is again a null
%sequence. Thus we have
%\begin{equation*}
%|\mathcal{T}_\delta| = 2\ln(2)\delta^{-2} + o(\delta^{-2})
%\hspace{2ex}\text{for $\delta \to 0$.}
%\end{equation*} 

\end{proof}

%-------------------------------------------------------------------
% CONSTRUCTION B
%-------------------------------------------------------------------

\section{Another Construction}

Let us consider another algorithm constructing $\delta$-bracketing
covers $\mathcal{Z}_\delta$ for anchored rectangles:

Let again $a_i = a_i(\delta) = (1-i\delta)^{1/2}$
for $i=0,\ldots, \zeta(\delta) := \lceil \delta^{-1} \rceil -1$,
and $a_{\zeta(\delta) +1} = 0$. Put $
\overline{a}_i := (a_i(\delta), a_i(\delta))$ for all $i$. 
We first decompose $[0,1]^2$ into layers 
$$
L^{(i)}(\delta) := 
[0,\overline{a}_i]\setminus [0,\overline{a}_{i+1}).
$$
Then, starting with $L^{(0)}(\delta)$, we will cover each layer 
$L^{(i)}(\delta)$ separately
with $\delta$-brackets. To this purpose we cover for fixed 
$i\in\{0,\ldots, \zeta(\delta)-1\}$ the stripe 
$[(0, a_{i+1}), (a_i, a_i)]$ recursively by the following procedure:

Put $\mathcal{S}_i(\delta) := \{ [\overline{a}_{i+1} ,
\overline{a}_{i}]\}$ and $x_1 := a_{i+1}(\delta)$.\\
If $x_{j}a_i(\delta) > 0$, then define 
$$
x_{j+1} := \max\{0, (x_j a_i(\delta) - \delta)/a_{i+1}(\delta) \}
$$
and put
$$
\mathcal{S}_i(\delta) := \mathcal{S}_i(\delta) \cup 
\{ [(x_{j+1}, a_{i+1}(\delta)), (x_j, a_i(\delta))] \}.
$$
If $x_j a_i(\delta) \le \delta$, then stop the covering procedure.

It is easy to see that for each $i$ the resulting set 
$\mathcal{S}_i(\delta)$
consists of $\delta$-brackets whose union is 
$[(0, a_{i+1}), (a_i, a_i)]$. In fact, we see that for $i=0$ the
$x_j$, $j=0,1,\ldots$, we get from the procedure above form
the projection of $\Gamma_\delta$ (defined as in (\ref{non_eq_grid})),
i.e., the set $\{x_0,x_1,\ldots, x_{\kappa(\delta)}, 
x_{\kappa(\delta)+1}\}$, where $x_{\kappa(\delta)} \le \delta$ and
$x_{\kappa(\delta)+1}=0$.
Thus $|\mathcal{S}_0(\delta)|=\kappa(\delta)+1$. Using the scaling
Lemma \ref{scaling} with $\lambda = (a_i^{-1}, a_i^{-1})$ we 
deduce that consequently 
$|\mathcal{S}_i(\delta)| = \kappa(\delta_i) + 1$, where 
$\delta_i := \delta/(1-i\delta)$.
(Observe that $a_{i+1}(\delta)/a_i(\delta)= a_1(\delta_i)$ for
$i<\zeta(\delta)$.)
By symmetry, we can cover $L^{(i)}(j)$ by
$2\kappa(\delta_i)+ 1$ $\delta$-brackets. (Observe that 
$L^{\zeta(\delta)}(\delta)$ is already a $\delta$-bracket.)
More precisely, using the mapping $\fs: \R^2\to \R^2$, 
$(x,y) \mapsto (y,x)$ we have
\begin{equation*}
\mathcal{Z}_\delta = \bigcup^{\zeta(\delta)}_{i=0}
\big( \mathcal{S}_i(\delta) \cup \fs(\mathcal{S}_i(\delta)) \big).
\end{equation*}
The set $\mathcal{Z}_\delta$ is a $\delta$-bracketing cover
of $[0,1]^2$ with 
\begin{equation}
\label{ZdeltaId}
|\mathcal{Z}_\delta| = \left( 
\sum^{\zeta(\delta)-1}_{i=0} (2\kappa(\delta_i) + 1)
\right) + 1.
\end{equation}
In Figure \ref{figg025} and \ref{figg005} we see $\mathcal{Z}_\delta$
for $\delta = 0.25$ and $0.05$.

\begin{figure}
\begin{center}
\hspace{1cm}{\epsfig{file=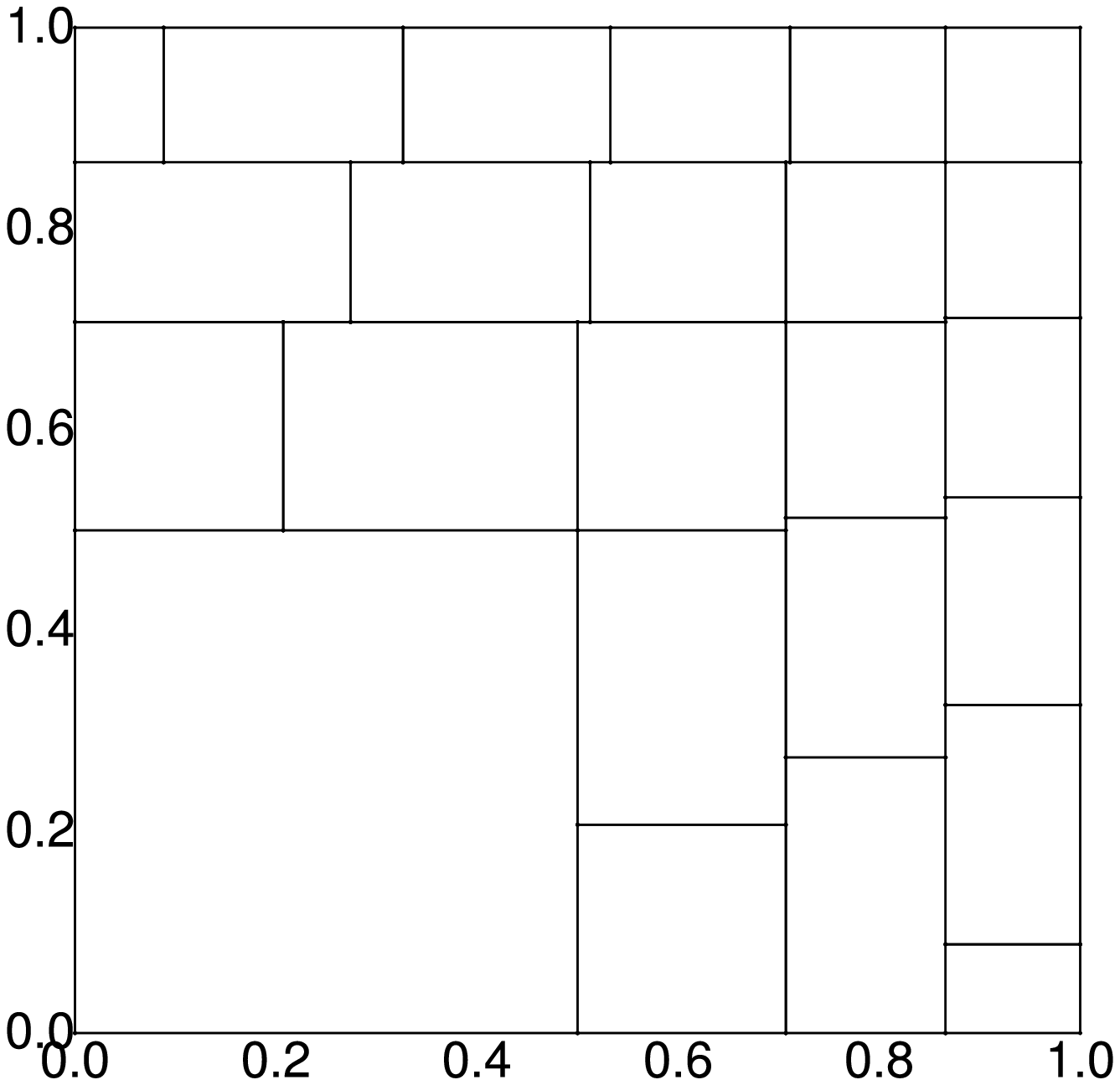,height
=.4\textheight
%,width=.45\textwidth
}}
\caption{\label{figg025}$\mathcal{Z}_\delta$ for
$\delta = 0.25$.}
\end{center}
\end{figure}  

\begin{figure}
\begin{center}
\hspace{1cm}{\epsfig{file=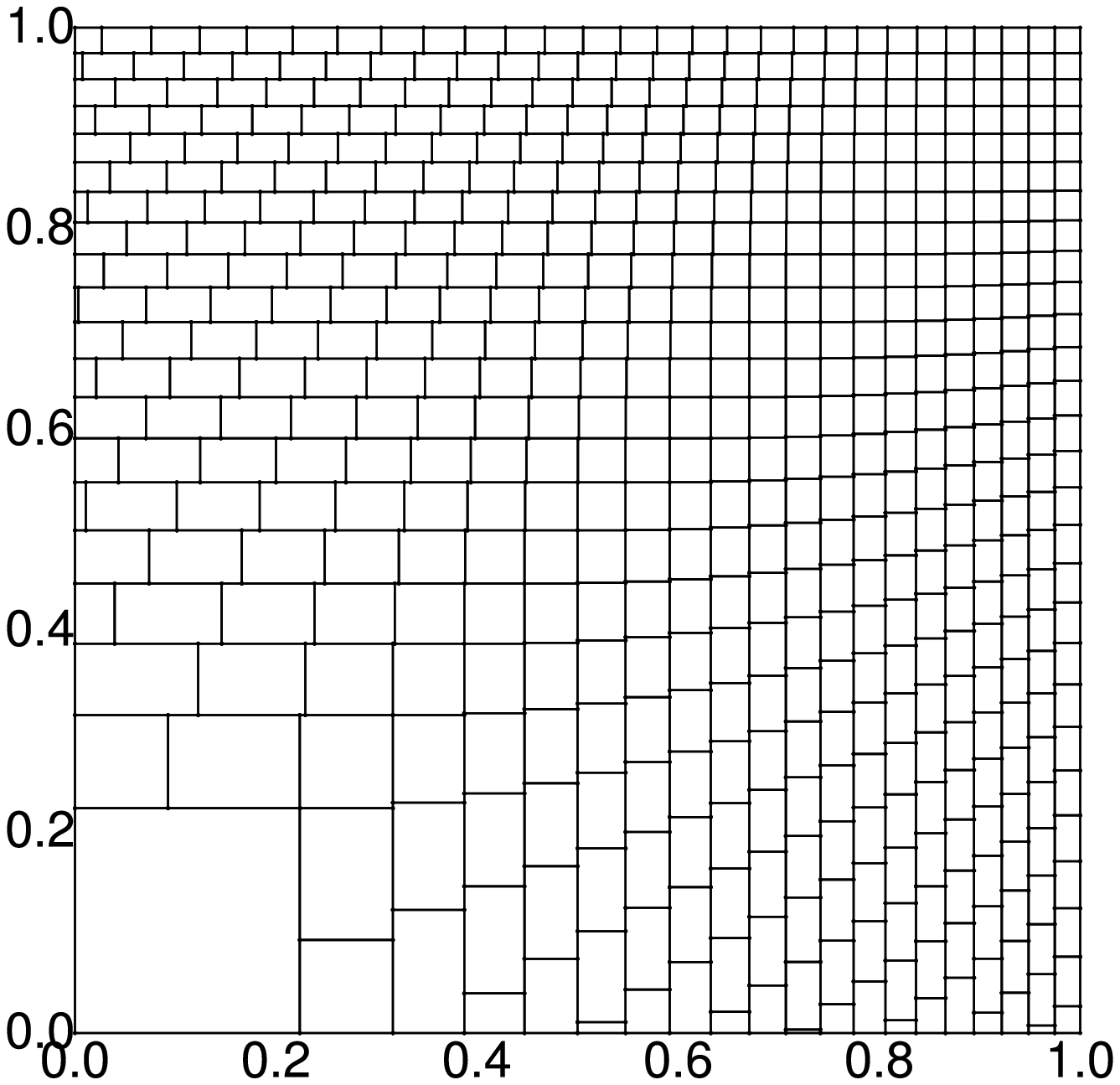,height
=.4\textheight
%,width=.45\textwidth
}}
\caption{\label{figg005}$\mathcal{Z}_\delta$ for
$\delta = 0.05$.}
\end{center}
\end{figure}  

%Now we want to determine the asymptotical behavior of 
%$|\mathcal{Z}_\delta|$ for $\delta$ tending to zero and
%calculate the coefficient in front of the dominating term
%$\delta^{-2}$.
%In \cite[Theorem 2.3]{DGS} we proved that 
%$\kappa_1(x) \le \lceil 2\ln(2)x^{-1} \rceil$ and that the quotient
%of the left and the right hand side of this inequality converges
%to $1$ as $\delta$ approaches $0$. Let us define
%$$
%R(x) = 2\,\left(
%\frac{\ln( 1- \sqrt{1-x} ) - \ln (x)}{\ln(1-x)} 
%- \frac{\ln(2)}{x} \right).
%$$
%Then we get 
%\begin{equation*}
%|\mathcal{Z}_\delta| = 
%4 \ln(2) \sum^{\kappa_2(\delta)-1}_{i=0} \delta_i^{-1}
%+ 2 \sum^{\kappa_2(\delta)-1}_{i=0} R(\delta_i) + O(\delta^{-1}).
%\end{equation*}
From the identity (\ref{ZdeltaId}), the definition of $\zeta(\delta)$,
and (\ref{kappaO1}) we get
\begin{equation*}
|\mathcal{Z}_\delta| = 2 \left( \sum^{\zeta(\delta)-1}_{i=0}
2\ln(2)\delta^{-1}_i \right) + O(\delta^{-1}).
\end{equation*}
Now $\delta_i^{-1} = \delta^{-1}-i$, hence
$$
\sum^{\zeta(\delta)-1}_{i=0} \delta_i^{-1}
= \zeta(\delta) \delta^{-1} -
\frac{\zeta(\delta) (\zeta(\delta)-1)}{2}
=\frac12 \delta^{-2} + O(\delta^{-1}).
$$
Thus
\begin{equation*}
|\mathcal{Z}_\delta| = 
2 \ln(2) \delta^{-2} + O(\delta^{-1}).
\end{equation*}
%We show now that the sum over the $R(\delta_i)$s is at most of 
%order $o(\delta^{-2})$:
%An elementary analysis reveals that the function
%$F: ]0,1[\to\R$, $x\mapsto x R(x)$ is monotonic 
%decreasing with $\lim_{x\to 0} F(x) = 0$. Thus we have
%$R(x) = o(x^{-1})$. Furthermore,
%$$
%\delta_i = \frac{\delta}{\delta(\delta^{-1} - i)} 
%\le \frac{1}{\lceil \delta^{-1} \rceil - (i+1)}
%$$
%for $i=0,\ldots, \kappa_2(\delta) -1$. Putting 
%$N = N(\delta) = \lceil \delta^{-1} \rceil - 1$, we get
%\begin{equation*}
%\sum^{\kappa_2(\delta)-1}_{i=0} |R(\delta_i)|
%= \sum^{\kappa_2(\delta)-1}_{i=0} \delta^{-1}_i
%|\delta_i R(\delta_i)|
%\le \delta^{-1} \sum^{\kappa_2(\delta)-1}_{i=0} |\delta_i R(\delta_i)|
%\le \delta^{-1} \sum^{N}_{k=1} \left| \frac{1}{k} R\left(
%\frac{1}{k} \right) \right|.
%\end{equation*}
%If we put $C_k := |\frac1k R(\frac1k)|$, 
%then $(C_k)$ is a null sequence and
%$$
%\delta^2 \sum^{\kappa_2(\delta)-1}_{i=0} |R(\delta_i)|
%\le \frac{1}{N} \sum^N_{k=1} C_k \to 0
%\hspace{2ex} \text{as $N\to\infty$,}
%$$
%since the Ces\`aro mean of any null sequence is also a null
%sequence. Thus we have
%\begin{equation*}
%|\mathcal{Z}_\delta| = 2\ln(2)\delta^{-2} + o(\delta^{-2})
%\hspace{2ex}\text{for $\delta \to 0$.}
%\end{equation*} 

Altogether, we proved the following proposition.

\begin{proposition}
For $\delta\in (0,1)$ 
the set of rectangles $\mathcal{Z}_\delta$ constructed
above is a $\delta$-bracketing cover of $\mathcal{C}_2$.
Its cardinality is given by
\begin{equation}
|\mathcal{Z}_\delta| =
\left( \sum^{\zeta(\delta)-1}_{i=0} (2\kappa(\delta_i)+1) 
\right) + 1 
= 2 \ln(2)\delta^{-2} + O(\delta^{-1}),
\end{equation}
where $\zeta(\delta) = \lceil \delta^{-1} \rceil - 1$, 
$\delta_i = \delta/(1-i\delta)$, and $\kappa(\delta_i)
= \kappa(\delta_i,2)$ as defined in (\ref{kappa}).
\end{proposition}

%----------------------------------------------------------
% CONSTRUCTION C (Reoriented Rectangles)
%----------------------------------------------------------

\section{Re-Orientation of the Brackets}

A positive aspect of the two previous constructions is that 
(essentially) all brackets in the resulting $\delta$-bracketing 
covers have largest possible weight $\delta$ and overlap only 
on sets of 
Lebesgue measure zero. But if we look at the brackets in 
Thi\'emard's construction which 
have some distance to the upper edge of the unit rectangle
$[0,1]^2$,
then these boxes do 
certainly  not satisfy the ``maximum area criterion'' stated
in Lemma \ref{OptBox}. The same holds for the brackets in 
$Z_\delta$ which are close to the $x$- or the $y$-axis and 
away from the main diagonal.
The idea of our next construction is to generate 
a bracketing cover similarly as in the previous section, but
to ``re-orientate'' the brackets from time to time in the course
of the algorithm to
enlarge the area which is covered by a single bracket.
Of course the algorithm should still be simple and 
avoid to much overlap of the generated brackets. 

Before stating the technical details, we want to present the 
underlying geometrical idea in a simplified way:

Like the construction in the previous section,
our new bracketing cover should be symmetric with respect to both
coordinate axes. Thus we only have to state explicitly how to cover the 
subset
$$
H := \{ (x,y)\in [0,1]^2 \,|\, x\le y \}
$$
of $[0,1]^2$. For a certain number $p=p(\delta)$ we then subdivide
$H$ into sectors
\begin{equation*}
T^{(h)} := \left\{ (x,y)\in H\setminus\{(0,0)\} \,\Bigg|\,
\frac{h-1}{2^p} \le \frac{x}{y} \le \frac{h}{2^p} \right\}
\cup \{(0,0)\},\hspace{2ex}h=1,\ldots,2^p.
\end{equation*}
In the same manner as we decomposed in the previous construction
the set $H$ into stripes $[(0,a_{i+1}), (a_i,a_i)]$, we now 
decompose $T^{(2^p)}$ into
stripes $[(0,a_{i+1}), (a_i,a_i)]\cap T^{(2^p)}$.
We do it similarly with the sectors $T^{(1)}, \ldots, T^{(2^p-1)}$,
but we use thicker (and therefore less) stripes there.
Covering each of these stripes by brackets whose height is exactly 
the height of the corresponding stripe, we see that each bracket
has almost the maximum possible area. Provided we can avoid to 
much overlap at the boundaries of the sectors, we thus need only a 
very small number of these brackets to cover $[0,1]^2$.  

Let us now state the generating algorithm precisely.
%(and let me apologize to the gentle reader for probably leaving
%him with the feeling that the technicalities are quite tedious)
We define ``discretized'' versions $T^{(h)}_{\dis}$ of the sectors
$T^{(h)}$, composed of stripes.
To this purpose we define for each $h\in \{1,\ldots,2^p\}$
\begin{equation*}
\rho^{(h)}(\delta) := \lceil h2^{-p}\delta^{-1} \rceil -1.
\end{equation*} 
(Note that $\rho^{(2^p)}(\delta)$ is precisely $\zeta(\delta)$
as defined in the previous section.)
For $i=0,\ldots,\rho^{(h)}(\delta)$ let
\begin{equation*}
T^{(h)}_i(\delta) :=
[(t^{(h)}_{i+1}(\delta), a_{i+1}^{(h)}(\delta)),
(h2^{-p} a_i^{(h)}(\delta), a^{(h)}_i(\delta))],
\end{equation*}
where
\begin{equation*}
t_{i+1}^{(h)}(\delta) :=
\frac{h-1}{2^p} \left( 1 - \left\lceil \frac{h-1}{h}i -\frac{1}{h}
\right\rceil \frac{2^p}{h-1} \delta \right)^{1/2}_+
\end{equation*}
and 
\begin{equation*}
a^{(h)}_i(\delta) := \left( 1 - i\frac{2^p}{h}\delta \right)^{1/2}_+.
\end{equation*}
(Here we use the convention to denote for a general function $f$
by $f_+$ the function $f 1_{f^{-1}([0,\infty))}$, where $1_A$ is the 
characteristic function of a set $A$. In particular we have
$t^{(1)}_{i+1}(\delta)=0$ for all $i$ and 
$a^{(h)}_{\rho^{(h)}(\delta)+1}(\delta) = 0$ for all $h$.)
We put 
%\begin{equation*}
$\overline{a}_i^{(h)}(\delta) :=
( h2^{-p} a^{(h)}_i(\delta), a^{(h)}_i(\delta) )$.
%\end{equation*}
Then 
$$
T^{(h)}_{\dis} := \bigcup^{\rho^{(h)}(\delta)}_{i=0}
T^{(h)}_i(\delta)
$$ 
can be viewed as a discretized version
(discretized with respect to a decomposition into stripes) of 
$T^{(h)}$.

Now for $h= 2^p, 2^p-1,\ldots,1$ we cover each stripe
$T^{(h)}_i(\delta)$, $i=0,1,\ldots, \rho^{(h)}(\delta)$, 
of the ``discretized'' sectors $T^{(h)}_{\dis}$ by brackets having 
exactly the height of the stripe $T^{(h)}_i(\delta)$
in the following manner:

\vskip 0.2cm \noindent
{\bf Algorithm} {\sc RE-ORIENTED BRACKETS}
\vskip 0.2cm \noindent
{\it Input:} $\delta \in (0,1)$, $p\in [0,\infty)$.
\vskip 0.1cm \noindent
{\it Output:} A $\delta$-bracketing cover $\mathcal{R}_\delta$.
\vskip 0.2cm \noindent
{\it Main}
\vskip 0.2cm

$\mathcal{R}_\delta:=\emptyset$
\vskip 0.2cm

\hspace{3ex}For $h=2^p$ to $1$
\vskip 0.2cm

\hspace{6ex}For $i=0,...,\rho^{(h)}(\delta)$
\vskip 0.2cm

\begin{equation*}
\begin{split}
\hspace{9ex}x_1 &:= \left( \frac{h}{2^p} (a^{(h)}_i(\delta))^2 
- \delta \right)
\left( a^{(h)}_{i+1}(\delta) \right)^{-1}
= \frac{h}{2^{p}} a^{(h)}_{i+1}(\delta)\\
\hspace{9ex}\mathcal{R}_\delta &:= \mathcal{R}_\delta \cup
\{[\overline{a}^{(h)}_{i+1}(\delta), \overline{a}_i^{(h)}(\delta)]\}
\end{split}
\end{equation*}
\vskip 0.2cm

\hspace{9ex}For $j=1,2,\ldots$
\vskip 0.2cm

\hspace{12ex}If $x_j > t^{(h)}_{i+1}(\delta)$
\vskip 0.2cm

\begin{equation*}
\begin{split}
\hspace{15ex}&x_{j+1} := \left(x_j a^{(h)}_i(\delta) - \delta \right)_+
\left( a^{(h)}_{i+1}(\delta) \right)^{-1}\\
\hspace{15ex}&\mathcal{R}_\delta := \mathcal{R}_\delta \cup
\{[(x_{j+1}, a^{(h)}_{i+1}(\delta)), (x_j, a_i^{(h)}(\delta))]\}
\end{split}
\end{equation*}
\vskip 0.2cm

\hspace{12ex}Else next $i$
\vskip 0.2cm

$\mathcal{R}_\delta := 
\mathcal{R}_\delta \cup \fs(\mathcal{R}_\delta)$ 
\vskip.2cm\noindent

The output set $\mathcal{R}_\delta$ is visualized in 
Figure \ref{figr005} and \ref{figr002} for $\delta = 0.05$ and
$\delta=0.02$; there we have chosen $p=p(\delta)$ to be
\begin{equation}
\label{pdelta}
p(\delta) = \max \left\{ 
\left\lfloor \frac{\ln(\delta^{-1}) - k}{c} \right\rfloor 
\,,\, 0 \right\}
\end{equation}
with $k=0$ and $c=1.7$. With this choice we get for $\delta =0.25$ that
$p=0$ and consequently $\mathcal{Z}_\delta = \mathcal{R}_\delta$; 
thus Figure \ref{figg025} shows $\mathcal{R}_\delta$ 
for $\delta=0.25$.

\begin{figure}
\begin{center}
\hspace{1cm}{\epsfig{file=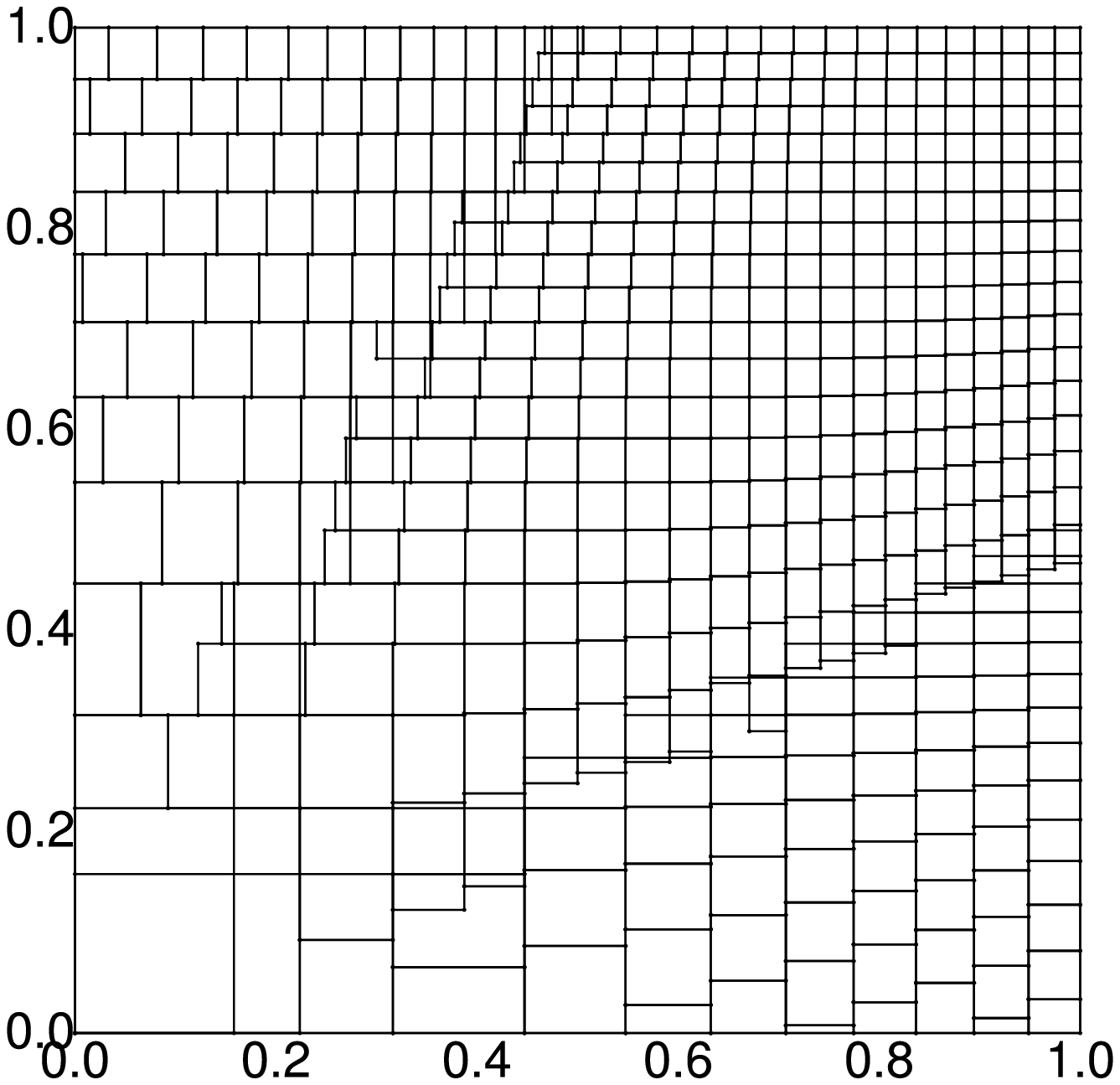,height=0.4\textheight
%, width=.45\textwidth
}}
\caption{\label{figr005}$\mathcal{R}_\delta$ for $\delta=0.05$.}
\end{center}
\end{figure}  

\begin{figure}
\begin{center}
\hspace{1cm}{\epsfig{file=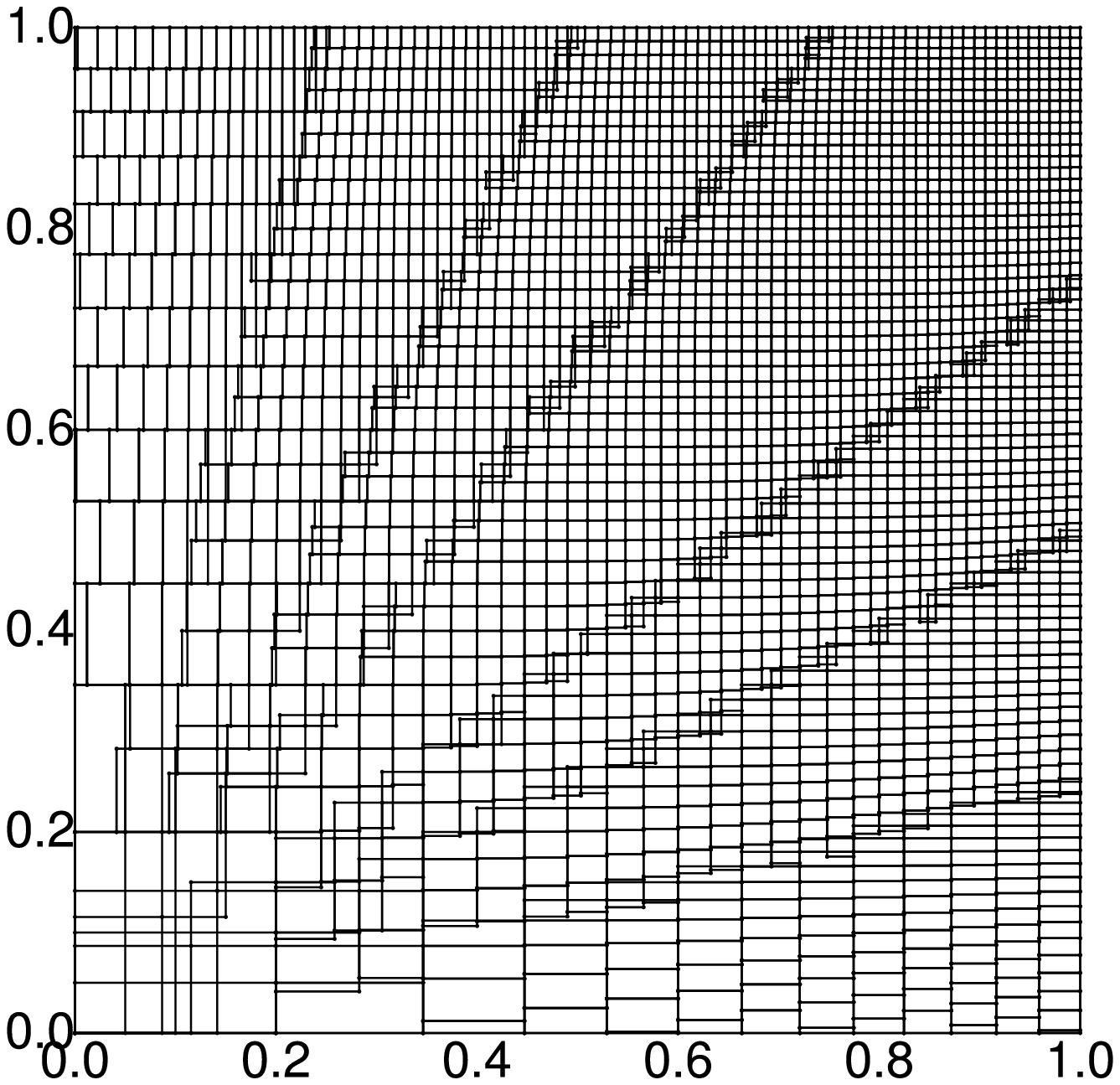,height=0.4\textheight
%, width=.45\textwidth
}}
\caption{\label{figr002}$\mathcal{R}_\delta$ for $\delta=0.02$.}
\end{center}
\end{figure}

Let us now prove the following proposition.

\begin{proposition}
\label{ThmReOr}
The output set $\mathcal{R}_\delta$ of the algorithm stated above is
a $\delta$-bracketing cover.
If $p=p(\delta)$ is a decreasing function on $(0,1)$ 
with $\lim_{\delta\to 0} p(\delta) = \infty$ and
$2^p = o(\delta^{-1})$ as $\delta$ tends to zero,
then the bracketing cover $\mathcal{R}_\delta$ satisfies
\begin{equation*}
|\mathcal{R}_\delta| = \delta^{-2} + o(\delta^{-2}).
\end{equation*}
\end{proposition}

\begin{proof}
One can check by direct calculation that all rectangles that
are added to $\mathcal{R}_\delta$ are in fact $\delta$-brackets.
The points $\overline{a}^{(h)}_i(\delta)$ are lying on the 
lines $\tfrac{x}{y} \equiv \tfrac{h}{2^p}$ and 
the $x$-coordinates $t^{(h)}_{i+1}(\delta)$ of the left corners
of the stripes $T^{(h)}_i(\delta)$ are chosen in such a way that
$H\subseteq \cup^{2^p}_{h=1} T^{(h)}_{\dis}$: For given
$h\ge 2$ and a given $i$ the index 
$$
j_i := \left\lceil \frac{h-1}{h} i - \frac{1}{h} \right\rceil
$$
is uniquely determined by
$$
a^{(h-1)}_{j_i}(\delta) > a^{(h)}_{i+1}(\delta) 
\ge a^{(h-1)}_{j_i + 1}(\delta),
$$
and we have
$$
t^{(h)}_{i+1}(\delta) = \frac{h-1}{2^p} a^{(h-1)}_{j_i}(\delta).
$$ 

Let now $\omega(\delta,h,i)$ be the minimal number of 
$\delta$-brackets of heights 
$a^{(h)}_i(\delta) - a^{(h)}_{i+1}(\delta)$ that we need to
cover $T^{(h)}_i(\delta)$. Using the 
scaling Lemma \ref{scaling} with 
$\lambda = (2^p(h a_{i}(\delta))^{-1},(a_{i}(\delta))^{-1})$ 
we see that
we have $\omega(\delta,h,i) = \omega(\delta^{(h)}_i, t(\delta,h,i))$
as defined in Section \ref{PRE}, where
$\delta^{(h)} = \tfrac{2^p}{h}\delta$ and, coinciding with the
convention from the previous section, 
$\delta^{(h)}_i = \delta^{(h)}/(1-i\delta^{(h)})$, and
\begin{equation*}
t(\delta,h,i) =
\frac{h-1}{h} \left( 1 - \left( 
\left\lceil \frac{h-1}{h} i - \frac{1}{h} \right\rceil 
\frac{h}{h-1} - i \right) \delta^{(h)}_i \right)^{1/2}_+.
\end{equation*}
Due to (\ref{omegadeltat}) we get
\begin{equation*}
\begin{split}
&\omega(\delta,h,i) = \\
&\left\lceil 2\,
\frac{\ln \big( 1-(1-\delta^{(h)}_i)^{1/2} \big) - 
\ln \big( t(\delta,h,i)(1-(1-\delta^{(h)}_i)^{-1/2}) + 
\delta^{(h)}_i (1-\delta^{(h)}_i)^{-1/2} \big)}
{\ln(1-\delta^{(h)}_i)} \right\rceil.
\end{split}
\end{equation*}
We claim that 
\begin{equation}
\label{omegaO1}
\omega(\delta,h,i) = 2\ln \left( 1+\frac{1}{h} \right)
\left( \delta^{(h)}_i \right)^{-1} + O(1)
\hspace{2ex}\text{as $\delta^{(h)}_i$ tends to zero.}
\end{equation}
According to (\ref{kappaO1}) this is true for $h=1$. In general it
follows from the inequalities (\ref{ln1}), (\ref{ln2}) and
\begin{equation*}
\ln \big( t(\delta,h,i)(1-(1-\delta^{(h)}_i)^{-1/2}) + 
\delta^{(h)}_i (1-\delta^{(h)}_i)^{-1/2} \big)
= \ln \left( 1+\frac{1}{h} \right) - \ln(2) + \ln(\delta^{(h)}_i)
+ O(\delta^{(h)}_i).
\end{equation*} 
We have
\begin{equation}
\label{cardR}
|\mathcal{R}_\delta| = 
\sum^{2^p}_{h=1} \left( \sum^{\rho^{(h)}(\delta)}_{i=0}
 2\omega(\delta,h,i) \right) - (\rho^{(2^p)}(\delta)+1);
\end{equation}
here we have to subtract the last term to avoid
double-counting of the $\delta$-brackets on the main diagonal of
$[0,1]^2$. According to (\ref{omegaO1}) we get
\begin{equation*}
\begin{split}
|\mathcal{R}_\delta| &= 
4\,\sum^{2^p}_{h=1} \sum^{\rho^{(h)}(\delta)}_{i=0}
\ln \left( 1 + \frac{1}{h} \right) \left( \frac{h}{2^p} 
\delta^{-1} - i \right) + o(\delta^{-2})\\
&= 4\,\sum^{2^p}_{h=1} 
\ln \left( 1 + \frac{1}{h} \right) 
\left( \frac{1}{2} \left( \frac{h}{2^p} \delta^{-1} \right)^2
+ O(\delta^{-1}) \right) + o(\delta^{-2})\\
&= 2 \left( \sum^{2^p}_{h=1} \ln \left( 1 + \frac{1}{h} \right)  
\left( \frac{h}{2^p} \right)^2 \right) \delta^{-2} + o(\delta^{-2}).
\end{split}
\end{equation*}
It remains to show that the sum in parentheses is of the form
$\tfrac12 + o(1)$ as $\delta$ tends to zero (and thus $p$ tends
to infinity). But this follows easily from the identity
\begin{equation*}
\ln \left( 1 + \frac{1}{h} \right) = h^{-1} - h^{-2}
\sum^{\infty}_{k=0}(-1)^{k}\frac{h^{-k}}{k+2}.
\end{equation*} 
\end{proof}

%----------------------------------------------------------
% NUMERICAL COMPARISON OF THE ALGORITHMS
%----------------------------------------------------------

\section{Numerical Comparison and Conclusion}

Let us now compare the cardinalities of the different 
constructions of $\delta$-bracketing covers for some values of 
$\delta$, see the table below. For the construction of 
$\mathcal{R}_\delta$ we have chosen $p=p(\delta)$ exactly as in
(\ref{pdelta}). Thus
$$
2^p \approx \delta^{-\frac{\ln(2)}{c}} \approx \delta^{-0.4}
\approx o(\delta^{-1}),
$$
and the conditions of Proposition \ref{ThmReOr} are clearly
satisfied. Note that $2\ln(2) = 1.386294...$ and 
$(2\ln(2))^2 = 1.921812...$. Thus the table underlines the 
dominance of the leading terms in the expansion of the cardinalities
of the $\delta$-bracketing covers with respect to $\delta^{-1}$.

\vspace{0.1cm}

\begin{center}
\begin{tabular}{|r|l|l|l|l|l|l|l|l|}\hline
$\delta$&0.25 & 0.1 & 0.05 & 0.01& 0.005& 0.001& 0.0005& 0.0001\\ 
\hline\hline
$|\mathcal{G}_\delta|$ & 36&196 &784& 19321 & 77284& 1923769& 7689529& 
192182769 \\ \hline
$|\mathcal{T}_\delta|$ & 25&142 &565& 13922 & 55575& 1386908& 5546403& 138635574 \\ \hline
$|\mathcal{Z}_\delta|$ & 24&146 &572& 13962 & 55650& 1387292& 5547174& 138639434 \\ \hline
$|\mathcal{R}_\delta|$ & 24&128 & 490& 10888& 42162& 1021122& 4055986& 100514774\\ \hline 
$\delta^{-2}$& 16&100&400&10000&40000&1000000&4000000&100000000\\ \hline
\end{tabular}
\end{center}

Altogether, we provided in this paper an explicit 
construction of a $\delta$-bracketing cover $\mathcal{R}_\delta$
of $\mathcal{C}_2$ which is optimal in the sense that the coefficient
in front of the most significant term $\delta^{-2}$ in the
expansion of $|\mathcal{R}_\delta|$ with respect to 
$\delta^{-1}$ is optimal. 

We compared $\mathcal{R}_\delta$ to its simplified version
$\mathcal{Z}_\delta$ (which does not ``re-orientate'' the
brackets) and known bracketing covers from \cite{DGS} and \cite{Thi}.

We conjecture that extending the idea of construction of 
$\mathcal{R}_\delta$ to arbitrary dimension $d$, one can
generate $\delta$-bracketing covers $\mathcal{R}^{(d)}_\delta$
of $\dcubes$ whose cardinality satisfies
$$
|\mathcal{R}^{(d)}_\delta| = \delta^{-d} + o_d(\delta^{-d})
\hspace{2ex}\text{as $\delta$ approaches zero}
$$
(here $o_d$ should emphasize that the implicit constants in the
$o$-notation may depend on $d$), i.e., has the best possible 
coefficient in front of the most significant term $\delta^{-d}$
in the expansion with respect to $\delta^{-1}$.

We suspect that a rigorous proof of the conjecture might be
rather technical and tedious. That is why we would find even a rigorous
analysis for $d=3$ or computational experiments for higher
dimension quite interesting.

\subsection*{Acknowledgment}
I would like to thank Torben Rabe for performing the numerical 
tests and providing the figures for this paper.


\begin{thebibliography}{AAA}
  

  \footnotesize  
  
  \bibitem{Bec}
   J.~Beck, Some upper bounds in the theory of irregularities of 
   distribution,
   Acta Arith. 44 (1984) 115-130.

  \bibitem{BeC}
   J.~Beck, W.~W.~Chen, Irregularities of Distribution,
   Cambridge University Press, Cambridge, 1987.
  
 
  \bibitem{Cha}
   B.~Chazelle, The Discrepancy Method, Cambridge University Press,
   New York, 2000.

  \bibitem{DeL}
   L.~Devroye, G.~Lugosi, Combinatorial Methods in Density Estimation,
   Springer series in statistics, Springer-Verlag, New York, 2001.   

  \bibitem{Dic}
   J.~Dick,
   A note on the existence of sequences with small star discrepancy,
   J.~Complexity 23 (2007) 649-652.

  \bibitem{DoG}
    B.~Doerr, M.~Gnewuch,
    Construction of low-discrepancy point sets of small size by 
    bracketing covers and dependent randomized rounding,
    in: A. Keller, S. Heinrich, H. Niederreiter (Eds.), Monte Carlo and 
    Quasi-Monte Carlo Methods 2006, 299-312, Springer, Berlin Heidelberg, 2008.

  \bibitem{DGKP}
    B.~Doerr, M.~Gnewuch, P.~Kritzer, F.~Pillichshammer,
    Component-by-component construction of small low-discrepancy 
    point sets, to appear in Monte Carlo Methods Appl..

  \bibitem{DGS}
    B.~Doerr, M.~Gnewuch, A.~Srivastav,
    Bounds and constructions for the star-discrepancy via 
    $\delta$-covers,
    J.~Complexity 21 (2005) 691-709.    


  \bibitem{DrT}
    M.~Drmota, R.~F.~Tichy,
    Sequences, Discrepancies and Applications, 
    Lecture Notes in Mathematics, vol. 1651,
    Springer, Berlin, 1997.

  \bibitem{AvLpEx}
    M.~Gnewuch,
    Bounds for the average $L^p$-extreme and the 
    $L^\infty$-extreme discrepancy,
    Electron.~J.~Combin. 12 (2005), Research Paper 54, 11 pp.
 
  \bibitem{Bra}
    M.~Gnewuch,
    Bracketing numbers for axis-parallel boxes and applications 
    to geometric discrepancy, J.~Complexity 24 (2008) 154-172.
  
  \bibitem{Hau}
    D.~Haussler,
    Sphere packing numbers for subsets of the Boolean $n$-cube
    with bounded Vapnik-Chervonenkis dimension, 
    J. Comb. Theory A 69 (1995) 217-232.

  %\bibitem{Hei}
    %S.~Heinrich.
    %{\sl Some open problems concerning the star-discrepancy}.
    %J. Complexity 19 (2003) 416-419.

  \bibitem{HNWW}
    S.~Heinrich, E.~Novak, G.~W.~Wasilkowski, H.~Wo\'{z}niakowski,
    The inverse of the star-discrepancy depends linearly on the
    dimension, Acta Arith. 96 (2001) 279-302.

  \bibitem{HSW}
    F.~J.~Hickernell, I.~H.~Sloan, G.~W.~Wasilkowski,
    On tractability of weighted integration over bounded and
    unbounded regions in $\R^s$, Math. Comp. 73 (2004)
    1885-1905. 

  \bibitem{Hin}
    A.~Hinrichs,
    Covering numbers, Vapnik-\v{C}ervonenkis classes and bounds for the 
    star-discrepancy, J. Complexity 20 (2004) 477-483.

  \bibitem{Mat}
    J.~Matou\v{s}ek,
    Geometric Discrepancy,
    Springer, Berlin, 1999.

  \bibitem{Mha}
    H.~N.~Mhaskar,
    On the tractability of multivariate integration and approximation
    by neural networks, J. Complexity 20 (2004) 561-590.

  \bibitem{Nie}
    H.~Niederreiter,
    Pseudo Number Generation and Quasi-Monte Carlo Methods,
    SIAM, Philadelphia, 1992.
   
  \bibitem{NoW}
   E.~Novak, H.~Wo\'{z}niakowski,
   When are integration and discrepancy tractable?,
   in: R.~A.~DeVore, A.~Iserles, E.~S\"uli (Eds.), 
   Foundations of Computational Mathematics,
   Cambridge University Press, 2001, pp. 211-266.

  \bibitem{Tal}
    M.~Talagrand,
    Sharper bounds for Gaussian and empirical processes,
    Ann. Prob. 22 (1994) 28-76.

   \bibitem{Thi0}
    E.~Thi\'emard,
    Computing bounds for the star discrepancy, 
    Computing 65 (2000) 169-186.
   
   \bibitem{Thi}
    E.~Thi\'emard,
    An algorithm to compute bounds for the star discrepancy,
    J. Complexity 17 (2001) 850-880.   

   \bibitem{vdVW}
    A.~W.~van der Vaart, J.~A.~Wellner,
    Weak Convergence and Empirical Processes,
    Springer Series in Statistics, Springer, New York, 1996. 

\end{thebibliography}
\end{document}